%% file: nonlocal_stokes.tex
\documentclass[final]{siamltex}

\usepackage{amsmath,tikz,latexsym}
\usepackage{amssymb}
\usepackage{graphicx}
\usepackage{color}
\usepackage{amsfonts}
\usepackage{amscd}
\usepackage{mathrsfs}
\usepackage{bm}
\usepackage{enumerate}
\usepackage{algorithmic, algorithm}

    \newcommand{\bigzerou}{\smash{\lower1.7ex\hbox{\bg 0}}}

    \newcommand{\R}{\mathbb{R}}
    \newcommand{\D}{\displaystyle}

    \newcommand{\Z}{\mathbb{Z}}

    \newcommand{\f}{\mathbf{f}}
    \newcommand{\x}{\mathbf{x}}
    \newcommand{\y}{\mathbf{y}}
    \newcommand{\s}{\mathbf{s}}
    
    \newcommand{\om}{\bm{\omega}}
    \newcommand{\xx}{\bm{\xi}}
    \newcommand{\uu}{\textbf{u}}
    \newcommand{\bb}{\textbf{b}}
    
    \newtheorem{assumption}{Assumption}
    \newtheorem{remark}[theorem]{Remark}

\title{
Nonlocal Stokes equation with relaxation on the divergence free equation\thanks{This 
        work of YZ and ZS  were supported by NSFC grant 11671005 and the work of
        QD was supported in part by
        NSF DMS-1719699 and Army Research Office (MURI W911NF-15-1-0562).}
}

\author{
Yajie Zhang
\thanks{Yau Mathematical Sciences Center, Tsinghua University, Beijing, China,
100084. \textit{Email: zhangyajie@tsinghua.edu.cn}
}
\and
Qiang Du
\thanks{Department of Applied Physics and Applied Mathematics, Columbia University, New York, NY, 10027, USA,
 \textit{Email: qd2125@columbia.edu}
}
\and
Zuoqiang Shi
\thanks{Department of Mathematical Sciences \& Yau Mathematical Sciences Center, Tsinghua University, Beijing, China,
100084. \textit{Email: zqshi@tsinghua.edu.cn}}
}

\begin{document}

\maketitle

\begin{abstract}

In this paper, we consider a new nonlocal approximation to the linear Stokes system with periodic boundary conditions in two and three dimensional spaces . A relaxation term is added to the equation of nonlocal divergence free equation, which is reminiscent to the relaxation of local Stokes equation with small artificial compressibility.  Our analysis shows that the well-posedness of the nonlocal system can be established under some mild assumptions on the kernel of nonlocal interactions.  Furthermore, the new nonlocal system converges to the conventional, local Stokes system in second order as the horizon parameter of the nonlocal interaction goes to zero. The study  provides more theoretical understanding to some numerical methods,  such as smoothed particle hydrodynamics, for simulating incompressible viscous flows. 

\end{abstract}


\section{Introduction}

Recently, nonlocal models and corresponding numerical methods have been given increasing attention \cite{Du-CBMS}. 
For example,  in solid
mechanics, the theory of peridynamics \cite{Sill00} has been shown to be an alternative to conventional models of elasticity and fracture mechanics. 
Many numerical methods have also been developed to compute peridynamic model based on solid mathematical analysis \cite{Du-SIAM,DJTZ13,DLZ13,MD15,TD14,ZD10}. 
In image processing and data analysis, nonlocal methods also achieve great successes 
\cite{belkin2003led,CLL15,Coifman05geometricdiffusions,Gu04,KLO17,LWYGL14,LZ17,MCL16,LDMM,Peyre09,RWP06,Lui11}.
In many applications, nonlocal models have also been used as integral approximation to the conventional PDE models.
One can find studies on the natural link between the nonlocal integral relaxations and numerical schemes like the smoothed particle
hydrodynamics (SPH) \cite{DT17,GM77,LL10,MD14,Mon05}. 
Originally developed for astrophysical applications \cite{GM77},  SPH has now become very popular in simulating complex flows. In this paper, we like to provide more theoretical understanding to SPH by constructing and analyzing a nonlocal approximation to the linear Stokes system. 
Previously, the mathematical analysis of
 nonlocal integral relaxations to
the linear steady Stokes system has been studied in 
 \cite{DT17,hwi19}
for incompressible viscous
flows.  In  \cite{DT17},  the nonlocal systems employed a nonlocal Laplacian operator
together with a nonlocal gradient and a nonlocal divergence, with the latter two
sharing an anti-symmetric nonlocal interaction kernel. 
The well-posedness of the nonlocal system was shown under the assumption that 
the kernel functions defining the nonlocal gradient and divergence 
have to satisfy some extra conditions. These conditions imply certain strengthened
nonlocal interactions for nearby points. Under such conditions, the spaces
corresponding to local and nonlocal
incompressible vector fields actually coincide with each other. A natural question is
if such strengthened interactions is truly necessary since the popular kernels 
used in many practical SPH implementation do not share such properties.
In \cite{hwi19}, it was shown that  one can get a well-posed nonlocal Stokes equation
with nonlocal gradient and nonlocal divergence defined by operators involving more
general kernels but with a non-radial symmetric interaction neighborhood. 
Both models, given respectively in 
 \cite{DT17} and \cite{hwi19}, replaced the local divergence operator in the
 continuity equation (divergence-free constraint) by different  nonlocal 
divergence operators,  Moreover, in both cases, the nonlocal Laplacian operator can
 be defined by the composition of the 
same pair of nonlocal divergence and nonlocal gradient operators used for the
incompressibility and  that acting on the pressure in the momentum equation.

In this paper, we consider an alternative nonlocal approximation by adding a relaxation term in the continuity equation, i.e., the divergence free or incompressibility constraint. Nonlocal gradient and divergence operators associated with 
anti-symmetric kernels are retained as in \cite{DT17}. However, a relaxation term is added,  given 
by the nonlocal Laplacian of the pressure
with a coefficient dependent on the range of nonlocal interactions. 
With the new relaxation term,  the assumptions on the kernel functions can also be relaxed, same as 
the conclusion in \cite{hwi19}, but without the need of introducing non-radial symmetric interaction neighborhood. 
The coefficient is properly chosen so that
in the local limit, the extra term vanishes and the local incompressibility condition is recovered. Furthermore, 
the extra term does not degrade the order of the approximations to the local limit. That is, 
the nonlocal model converges to the original local Stokes system in second order accuracy with respect to the nonlocal horizon parameter.
Thus, we provide justifications to the use of more general kernels for incompressible SPH implementations,
although this must be accompanied by the nonlocal relaxation to the incompressibility.
For simplicity, we follow earlier works by considering problems defined in a periodic cell. The more general
cases involving various other popular boundary conditions will be studied in upcoming works.

The rest of the paper is organized as follows. The detailed models
and notation are presented in Section \ref{sec:model}. The well-posedness of the
nonlocal model is established in Section \ref{sec:WellPosedness}. The vanishing nonlocality limit is
 analyzed in Section \ref{sec:van} . We prove that the solution of the 
nonlocal system converges to the solution of the original Stokes system as $\delta$ goes to 0 and the rate is second order. Furthermore, we consider modifications of the nonlocal model in order to get higher
regularity in Section \ref{sec:reg}. Finite dimensional approximations of the nonlocal model are studies
in Section \ref{sec:num} based on a Fourier spectral discretization.
Finally, in Section \ref{sec:con}, we conclude with a summary and a
discussion in future research.

\section{Model description}
\label{sec:model}


We consider the linear Stokes system with a periodic boundary condition. 
Let $\textbf{u}$ be the velocity field, $p$ the pressure, $\textbf{f}$ the body force, 
 $\Omega=(-\pi, \pi )^d \subset \R^d$ a bounded domain with a boundary $\partial \Omega$, $d=2$ or $3$. 
 The classical, local Stokes equations of interests here refer to the system
\begin{equation} \label{eq:stokes}
\begin{cases}
- \Delta \textbf{u} + \nabla p= \textbf{f}, & \x \in \ \Omega, \\
\nabla \cdot \uu =0, & \x\in \ \Omega.
\end{cases}
\end{equation}
We also impose average zero condition the remove the ambiguity of the constant 
\begin{align}
  \label{eq:velocity-ave}
  \int_\Omega \uu (\x) d \x=0
\end{align}
and
\begin{align}
  \label{eq:pressure-ave}
  \int_\Omega p(\x) d \x=0.
\end{align}
In \cite{DT17}, a nonlocal version of the above Stokes system with periodic boundary condition has been studied. The nonlocal model proposed there is
\begin{equation} \label{mod:du}
\begin{cases}
- \mathcal{L}_{\delta} \textbf{u}_{\delta}(\x) + \mathcal{G}_{\delta}  p_{\delta} (\x)= \textbf{f} (\x), & \x \in \Omega, \\
\mathcal{D}_{\delta} \uu_{\delta} (\x)  =  0, & \x \in \Omega,
\end{cases}
\end{equation}
The associated nonlocal operators are the nonlocal diffusion (Laplacian) operator $\mathcal{L}_{\delta}$, nonlocal gradient operator $\mathcal{G}_{\delta}$, and
nonlocal divergence operator $\mathcal{D}_{\delta}$,  given respectively by
\begin{equation}
\label{def:L}
\mathcal{L}_{\delta} \textbf{u}(\x)=\int_{\R^d} \omega_{\delta}(|\y-\x|)(\uu(\y)-\uu(\x))d \y, 
\end{equation}
\begin{equation}\label{def:G}
\mathcal{G}_{\delta} p(\x)=\int_{\R^d} \om_{\delta}(\y-\x)(p(\y)-p(\x))d \y,
\end{equation}
\begin{equation}\label{def:D}
\mathcal{D}_{\delta} \textbf{u}(\x)=\int_{\R^d} \om_{\delta}(\y-\x) \cdot (\uu(\y)-\uu(\x))d \y,
\end{equation}
The above operators are determined by a nonlocal scalar-valued kernels $\omega_{\delta}$ and a vector-valued kernel $\om_{\delta}$. 
The vector-valued kernel has a special form $\om_{\delta}(\x)= \hat{\omega}_{\delta} (| \x|) \x / |\x|$. 
Here $\delta>0$ is a parameter that characterizes the range of nonlocal intersection. 

Since the boundary condition is assumed to be periodic and nonlocal operators are all of the convolution type, 
it is convenient to analyze the nonlcoal model via Fourier analysis.
 In Fourier space, nonlocal model \eqref{mod:du} becomes 
\begin{equation} \label{mod:du-fft}
\left( \ \begin{matrix}
\D \lambda_{\delta}(\xx)I_d & i \bb_{\delta}(\xx) \\ i(\bb_{\delta}(\xx))^T & 0
\end{matrix} \ \right) \left( \ \begin{matrix}
\hat{\uu}_{\delta}(\xx) \\ \hat{p}_{\delta}(\xx)
\end{matrix} \ \right)=\left( \ \begin{matrix}
\hat{\f}(\xx) \\ 0
\end{matrix} \ \right)
\end{equation}
where $i=\sqrt{-1}$, $\lambda_{\delta}(\xx)$ and $\bb_{\delta}(\xx)$ are Fourier symbols of $\mathcal{L}_\delta$ and $\mathcal{G}_{\delta}$ respectively.
It is easy to check that the matrix in \eqref{mod:du-fft} is invertible if and only if $\bb_{\delta}(\xx)\ne 0$. Then we have to impose proper conditions on $\hat{\omega}$ 
to guarantee $\bb_{\delta}(\xx)\ne 0,\; \forall \xx$. The conditions proposed in \cite{DT17} is that
\begin{itemize} 
\item[1.] $r^{d-1}\hat{\omega}(r)$ is nonincreasing for $r \in (0,1)$;
\item[2.] $\hat{\omega}(r)$ is of fractional type in at least a small neighborhood of origin, namely, there exists $\epsilon>0$ such that for $s\in (0,\epsilon)$,
$$\hat{\omega}(r)=\frac{c}{r^{d+\beta}},$$
for some constant $c>0$ and $\beta\in (-1,1)$. 
\end{itemize}
Under above assumptions, $\hat{\omega}$ blows up near the origin. This singularity may introduce extra complications in the computation. To avoid such assumptions, \cite{hwi19} studied a similar nonlocal system \eqref{mod:du}, but with domains
of nonlocal interactions for the nonlocal
operators  $\mathcal{G}_{\delta}$, and
 $\mathcal{D}_{\delta}$ changed to non-radial ones (half disks  for $d=2$ and half spheres for $d-3$, as examples).
  
In this paper, we consider another nonlocal model to remove the singular assumption of the kernel function. The idea is to add an extra relaxation term in the divergence free equation. Namely, 
in our new model, the nonlocal counterpart of the divergence free equation is changed to
 \begin{equation} 
\mathcal{D}_{\delta} \uu_{\delta} (\x) - \delta^2 \mathcal{R}_{\delta} p_{\delta}(\x) = 0.
\end{equation}
 $\mathcal{R}_{\delta}$ is another nonlocal diffusion operator
\begin{equation}
\label{def:R}
\mathcal{R}_{\delta} p(\x)=\int_{\R^d} \tilde{\omega}_{\delta}(|\y-\x|)(p(\y)-p(\x))d \y,
\end{equation}
with a given kernel function $\tilde{\omega}_{\delta}$. The factor $\delta^2$ is small as $\delta \to 0$, and it is
used to recover the local incompressibility condition and to preserve the second order accuracy of the nonlocal approximation.

With the relaxation term in the divergence free equation, in the Fourier space, the nonlocal model becomes
\begin{equation} \label{mod:relax-fft}
\left( \ \begin{matrix}
\D \lambda_{\delta}(\xx)I_d & i \bb_{\delta}(\xx) \\ i(\bb_{\delta}(\xx))^T & \delta^2 c_{\delta}(\xx)
\end{matrix} \ \right) \left( \ \begin{matrix}
\hat{\uu}_{\delta}(\xx) \\ \hat{p}_{\delta}(\xx)
\end{matrix} \ \right)=\left( \ \begin{matrix}
\hat{\f}(\xx) \\ 0
\end{matrix} \ \right)
\end{equation}
with $c_{\delta}(\xx)$ is the Fourier symbol of $\mathcal{R}_{\delta}$. This linear system is always invertible such that the assumption on the kernel function in \cite{DT17} is
removed. 

For the sake of generality, 
in this paper, we consider a more general divergence equation instead of the divergence free condition, i.e.
\begin{equation}
    \label{eq:div}
    \nabla \cdot \uu(\x) =g(\x),\quad \x\in \Omega.
\end{equation}
with $g$ being a given function. Here, we require that $\textbf{f}$ and $g$ have mean zero to assure the compatibility, i.e.
\begin{equation}\label{eq:ave-free}
\int_{\Omega} \f(\x) d\x =\textbf{0},\quad \int_{\Omega} g(\x) d\x=0.
\end{equation}
We define a nonlocal model as follows:
\begin{equation} \label{a6}
\begin{cases}
- \mathcal{L}_{\delta} \textbf{u}_{\delta}(\x) + \mathcal{G}_{\delta}  p_{\delta} (\x)= \textbf{f} (\x), & \x \in \Omega, \\
\mathcal{D}_{\delta} \uu_{\delta} (\x) - \delta^2 \mathcal{R}_{\delta} p_{\delta}(\x) =  g(\x), & \x \in \Omega,
\end{cases}
\end{equation}
with normalization conditions on $\uu_{\delta}$ and $p_{\delta}$ to eliminate constant shifts,
\begin{equation} \label{a1}
\int_{\Omega} \uu_{\delta} (\x) d \x=0, \qquad \int_{\Omega} p_{\delta} (\x) d \x=0.
\end{equation}
The nonlocal operators used above are those defined in \eqref{def:L}, \eqref{def:G}, \eqref{def:D} and \eqref{def:R}.
We assume that $\omega_{\delta}$, $\hat{\omega}_{\delta}$ and $\tilde{\omega}_{\delta}$ are all non-negative, radial symmetric and with a compact support in the $\delta$ neighborhood $B(0, \delta)$ of the origin.


\section{Well-posedness of the nonlocal Stokes system} \label{sec:WellPosedness}

In this section, we study the well-posedness of the nonlocal system \eqref{a6}. 
First, we make some assumptions on the kernels $\omega_{\delta}=\omega_{\delta}(|\x|)$, $\tilde{\omega}_{\delta}=\tilde{\omega}_{\delta}(|\x|)$ and 
$\hat{\omega}_{\delta}=\hat{\omega}_{\delta}(|\x|)$ and introduce some notations.

\subsection{Assumption and notation}
Throughout this paper,  the kernel functions are assumed to satisfy.
\begin{assumption} \label{a5}
\begin{enumerate} 
\item  $\omega_\delta, \tilde{\omega}_\delta, \hat{\omega}_\delta$ have compact support in the sphere $B(0, \delta)$ and satisfy the normalization conditions:
\begin{equation} \label{a24}
\begin{split}
\frac{1}{2} \int_{\R^d} \omega_{\delta} (|\x|) |\x|^2 d \x = \int_{\R^d} \hat{\omega}_{\delta} (|\x|) |\x| d \x=d; \\
m \leq  \int_{\R^d} \tilde{\omega}_{\delta} (|\x|) |\x|^2 d \x \leq M
\end{split}
\end{equation}
for some constant $ \ 0<m<M$.
\item $\omega_\delta, \tilde{\omega}_\delta, \hat{\omega}_\delta$ are rescaled from kernels $\omega$, $\tilde{\omega}$ and $\hat{\omega}$ that have compact support in the unit sphere
\end{enumerate}
\begin{equation}
\mbox{ }\qquad
\omega_{\delta} (|\x|)=\frac{1}{\delta^{d+1}} \omega (\frac{|\x|}{\delta}), \;\; \tilde{\omega}_{\delta} (|\x|)=\frac{1}{\delta^{d+1}} \tilde{\omega} (\frac{|\x|}{\delta}), \;\; \hat{\omega}_{\delta} (|\x|)=\frac{1}{\delta^{d+2}} \hat{\omega} (\frac{|\x|}{\delta}).
\end{equation}
\end{assumption}

\begin{remark} 
Actually, we can allow $\omega_{\delta}$, $\tilde{\omega}_{\delta}$ and $\hat{\omega}_{\delta}$ to have non-compact support. The condition of compact support can be relaxed to, for some constant $C>0$,  
\begin{equation}
\int_{1/\epsilon}^{\infty} \omega(|\x|) |\x|^2 d \x \leq C \epsilon^2, \;\; \int_{1/\epsilon}^{\infty} \hat{\omega}(|\x|) |\x| d \x \leq C \epsilon^2, \;\; \int_{1/\epsilon}^{\infty} \tilde{\omega}(|\x|) |\x|^2 d \x \leq C \epsilon^2,
\end{equation}
for any $\epsilon>0$. Under such condition, the same theoretical results can be established.
\end{remark}

\begin{remark}  
The removal of the assumptions on $\hat{\omega}$ in \cite{DT17} widely broaden the choice of $\hat{\omega}$. For instance, we can simply choose a constant function $\hat{\omega} \equiv 3/\pi$ inside the unit disk and vanishes outside, or certain smooth trigonometric functions that are commonly used in the nonlocal theory. This greatly reduces the complication of the nonlocal model since none of the kernel functions necessarily have singularity near the origin and we can still expect a uniform bound on their corresponding Fourier symbols.

\end{remark}




Similarly as in \cite{DT17,hwi19}, we utilize the Fourier space representation of the periodic functions.
That is, 
\begin{equation}
\uu(\x)=\frac{1}{(2\pi)^d} \sum_{\xx \in \Z^d, \xx \neq 0} \hat{\uu}(\xx)e^{i \xx \cdot \x},
\quad
p(\x)=\frac{1}{(2\pi)^d} \sum_{\xx \in \Z^d, \xx \neq 0} \hat{p}(\xx)e^{i \xx \cdot \x},\nonumber
\end{equation}
where 
\begin{equation}
\hat{\uu}(\xx)=\int_{\Omega} \uu(\x) e^{-i \xx \cdot \x} d \x,
\quad
\hat{p}(\xx)=\int_{\Omega} p(\x) e^{-i \xx \cdot \x} d \x.\nonumber
\end{equation}

It is easy to check that 
the Fourier symbols  of operators $\mathcal{L}_{\delta}$, $\mathcal{G}_{\delta}$, $\mathcal{D}_{\delta}$ and $\mathcal{R}_{\delta}$ are given by (see Lemma 2 in \cite{DT17}):
\begin{equation}
\widehat{\mathcal{L}_{\delta} \uu} (\xx)=- \lambda_{\delta} (\xx) \hat{\uu} (\xx),
\quad
\widehat{\mathcal{G}_{\delta} p} (\xx)=i \bb_{\delta} (\xx) \hat{p} (\xx),\nonumber
\end{equation}

\begin{equation}
\widehat{\mathcal{D}_{\delta} \uu} (\xx)= i (\bb_{\delta} (\xx))^T \hat{\uu} (\xx),
\quad
\widehat{\mathcal{R}_{\delta} p} (\xx)=- c_{\delta} (\xx) \hat{p} (\xx),\nonumber
\end{equation}
where $\lambda_{\delta}$, $\bb_{\delta}$ and $c_{\delta}$ are given by
\begin{equation}
\lambda_{\delta}(\xx)=\int_{|\s| \leq \delta} \omega_{\delta}(|\s|)(1-\cos( \xx \cdot \s)) d \s,\nonumber
\end{equation}
\begin{equation}
\bb_{\delta}(\xx)=\int_{|\s| \leq \delta} \hat{\omega}_{\delta}(|\s|) \frac{\s}{|\s|} \sin( \xx \cdot \s) d \s,\nonumber
\end{equation}
\begin{equation}
c_{\delta}(\xx)=\int_{|\s| \leq \delta} \tilde{\omega}_{\delta}(|\s|)(1-\cos( \xx \cdot \s)) d \s.\nonumber
\end{equation}

The key to establish the well-posedness of the nonlocal system 
is to bound $\lambda_{\delta}(\xx),\;|\bb_{\delta} (\xx)|$ and $c_{\delta}(\xx)$, particularly from below.
\begin{proposition}\label{fft-symbol}
  Under the assumption $\ref{a5}$, we have
  \begin{itemize}
  \item[1.] For $\delta |\xx|> 1$,
    \begin{align*}
     &    |\xx|^2 \geq   \lambda_{\delta}(\xx) \geq  \frac{1}{16 \ \delta^2},\\
      & M  |\xx|^2 \geq  c_{\delta}(\xx) \geq   \frac{m}{16 \ \delta^2}.
    \end{align*}
\item[2.] For $0<\delta |\xx|\le 1$,
    \begin{align*}
&      \frac{1}{2}|\xx|^2<  |\xx|^2- \frac{|\xx|^2}{20}  \ ( \delta |\xx|)^2\le \lambda_{\delta}(\xx) \le  |\xx|^2,\\
& \frac{m}{2}|\xx|^2\le c_{\delta}(\xx) \le  M|\xx|^2,\\
&  |\xx |\ge | \bb_{\delta} (\xx) |\ge |\xx|-\frac{(\delta |\xx|)^2}{10} |\xx|> \frac{1}{2}|\xx|.
    \end{align*}
  \end{itemize}
\end{proposition}
The proof can be found in Appendix A.
\begin{remark}
The estimates in the above proposition imply that, when $\delta |\xx|>1$, $\lambda_{\delta}(\xx)$ has a uniform lower bound $\frac{1}{16 \delta^2} $ that is independent of its corresponding kernel $\omega(\x)$. This bound can be very useful in
the analysis of the convergence order of our nonlocal model to its local limit, which is studies in the next section.
Similar bounds have also been derived in the literature, for example, 
the same lower bound was shown in \cite{DY5}  
for the case that $r^d \omega(r)$ is decreasing on $r \in (0, \infty)$. For more general $\omega$, a lower bound for $\lambda_{\delta}(\xx)$ that depends on the choice of $\omega$ was given in the proof of Proposition $9$ of \cite{DT17}. Here, we get the uniform bound independent on $\omega$ under the very general assumption $\ref{a5}$.
\end{remark}

\subsection{Well-posedness}
First, we transform the nonlocal system to the Fourier space
\begin{equation} \label{a3}
A_{\delta}(\xx) \left( \ \begin{matrix}
\hat{\uu}_{\delta}(\xx) \\ \hat{p}_{\delta}(\xx)
\end{matrix} \ \right)=\left( \ \begin{matrix}
\hat{\f}(\xx) \\ \hat{g}(\xx)
\end{matrix} \ \right)
\end{equation}
where
\begin{equation}
A_{\delta}(\xx)= \left( \ \begin{matrix}
\D \lambda_{\delta}(\xx)I_d & i \bb_{\delta}(\xx) \\ i(\bb_{\delta}(\xx))^T & \delta^2 c_{\delta}(\xx)
\end{matrix} \ \right).\nonumber
\end{equation}

For notation convenience, let us define $q_\delta(\xx)=| \bb_{\delta} (\xx) |^2+ \delta^2 c_{\delta}(\xx) \lambda_{\delta}(\xx)$.
With both $\lambda_{\delta}(\xx)$ and $c_{\delta}(\xx)$ positive for each fixed $\xx\neq 0$,  we have 
$q_\delta (\xx)$ also positive. Thus, 
the inverse of $A_{\delta}(\xx)$ exists and has a closed form 
\begin{equation} \label{a7}
(A_{\delta}(\xx))^{-1}=\left( \ \begin{matrix}
\D\frac{1}{\lambda_{\delta}(\xx)}(I_d-\frac{\bb_{\delta}(\xx) \ \otimes \ \bb_{\delta} (\xx)}{q_\delta(\xx)  }) & 
\D  \frac{-i\bb_{\delta}(\xx)}{q_\delta(\xx) } \\[.4cm] \D  \frac{-i(\bb_{\delta}(\xx))^T}{q_\delta(\xx) }  &\D \frac{ \lambda_{\delta}(\xx)}{q_\delta(\xx) }
\end{matrix} \ \right).
\end{equation}
This implies that
\begin{align} \label{a8}
\hat{\uu}_{\delta}(\xx) &=\frac{1}{\lambda_{\delta}(\xx)}(I_d-\frac{\bb_{\delta}(\xx) \otimes \bb_{\delta} (\xx)}{
q_\delta(\xx) }) \hat{\f}(\xx) \ - \ \frac{ i \ (\bb_{\delta}(\xx))^T}{q_\delta(\xx) } \hat{g}(\xx),
\\ \label{a9}
\hat{p}_{\delta}(\xx) &=- \frac{i(\bb_{\delta}(\xx))^T}{q_\delta(\xx) } \hat{\f}(\xx) + \frac{ \lambda_{\delta}(\xx)}{q_\delta(\xx) } \hat{g}(\xx).
\end{align}

To prove the well-posedness and to obtain the desired a priori bounds, 
we need to get bound on $(A_{\delta}(\xx))^{-1}$, which is, equivalently, to bound the coefficients in front of $\hat{\f}(\xx)$ and $\hat{g}(\xx)$ uniformly in $\xx$ for $\xx\neq 0$. 

First, from the average free condition \eqref{eq:ave-free}, $\hat{\f}(0)=0$ and $\hat{g}(0)=0$, so we only need to consider the case $|\xx|>0$.

For the high frequency components such that $\delta|\xx|>1$, we have
\begin{align}
& \left| \frac{1}{\lambda_{\delta}(\xx)} \ \frac{\bb_{\delta}(\xx) \otimes \bb_{\delta} (\xx)}{
q_\delta(\xx)  } \right| = \frac{1}{\lambda_{\delta}(\xx)} \ \frac{| \bb_{\delta} (\xx) |^2}{
q_\delta(\xx) } \leq  \frac{1}{\lambda_{\delta}(\xx)}
\le 16\delta^2,\nonumber 
\\[.2cm]
&
 \frac{ | (\bb_{\delta}(\xx))^T |}{q_\delta(\xx) } \leq \frac{ | (\bb_{\delta}(\xx))^T | }{ 2 \delta \sqrt{c_{\delta}(\xx) \lambda_{\delta}(\xx)} \ | \bb_{\delta} (\xx) |  }=\frac{1}{2 \delta \  \sqrt{c_{\delta}(\xx) \lambda_{\delta}(\xx)} } \leq \frac{8}{\sqrt{m}} \delta,\nonumber 
\\[.2cm]
&
\frac{ \lambda_{\delta}(\xx)}{q_\delta(\xx) } \leq 
\frac{ \lambda_{\delta}(\xx)}{ \delta^2 c_{\delta}(\xx) \lambda_{\delta}(\xx) } \leq \frac{16}{m}.\nonumber 
\end{align}
Note that in the above and hereinafter, the matrix norm corresponds to that induced by the vector Euclidean norm.

Likewise, for the low frequency components, $0<\delta|\xx|\le 1$,
\begin{align}
&\left| \frac{1}{\lambda_{\delta}(\xx)} \ \frac{\bb_{\delta}(\xx) \otimes \bb_{\delta} (\xx)}{
q_\delta(\xx)  } \right|  \leq  \frac{1}{\lambda_{\delta}(\xx)}
\le \frac{2}{|\xx|^2} ,\nonumber \\[.2cm]
& \frac{ | (\bb_{\delta}(\xx))^T |}{q_\delta(\xx)  } \leq \frac{1}{| \bb_{\delta} (\xx) |} \leq \frac{2}{ |\xx|}, \nonumber \\[.2cm]
&\frac{ \lambda_{\delta}(\xx)}{q_\delta(\xx) } \leq  \frac{ \lambda_{\delta}(\xx)}{| \bb_{\delta} (\xx) |^2 }  \leq \frac{|\xx|^2}{(\frac{1}{2} |\xx|)^2}=4.\nonumber 
\end{align}

Above estimates directly imply the following theorem on the well-posedness of the nonlocal system.
\begin{theorem} \label{a37}
Assume that the kernels $\omega_{\delta}$, $\tilde{\omega}_{\delta}$ and $\hat{\omega}_{\delta}$ satisfy the assumption $\ref{a5}$. Given $0<\delta<1$, there exists a unique solution $(\uu_{\delta}, p_{\delta})$ to the nonlocal Stokes system $\eqref{a6}$ with periodic boundary condition given in the form of their Fourier series with $(\hat{\uu}_{\delta}(\xx), \hat{p}_{\delta}(\xx))$ through
\begin{equation}
 \left( \ \begin{matrix}
\hat{\uu}_{\delta}(\xx) \\ \hat{p}_{\delta}(\xx)
\end{matrix} \ \right)=(A_{\delta}(\xx))^{-1}
 \left( \ \begin{matrix}
\hat{\f}(\xx) \\ \hat{g}(\xx)
\end{matrix} \ \right)\nonumber
\end{equation}
where $(A_{\delta}(\xx))^{-1}$ is defined by $\eqref{a7}$. In addition, with $C>0$ depend only on $m$ and independent of $\delta$, $\f$, $g$, $\omega_{\delta}$, $\tilde{\omega}$, we have
\begin{equation} \label{a10}
\left \lVert \uu_{\delta} \right \rVert_{[L^2(\Omega)]^d} \leq C( \left \lVert \f \right \rVert_{[L^2(\Omega)]^d} + \left \lVert g \right \rVert_{L^2(\Omega)}),
\end{equation}
and
\begin{equation}  \label{a11}
\left \lVert p_{\delta} \right \rVert_{L^2(\Omega)} \leq C (\left \lVert \f \right \rVert_{[L^2(\Omega)]^d}+\left \lVert g \right \rVert_{L^2(\Omega)}).
\end{equation}
\end{theorem}









\subsection{Regularity pick-up}
\label{sec:pic}
In the original Stokes system, the solution is more regular than the right hand side $\f$. However, in the nonlocal system, if the kernel function 
$\omega(|\x|)$ is only assumed to be integrable, we can not get the regularity lift.  Because in this case, the high frequency components of $\mathcal{L}_\delta$ is bounded above rather than 
growing like $|\xx|^2$. This observation is similar to the case of other nonlocal elliptic problems \cite{DY5,DT17}.
On the other hand, this fact tells us that to get regularity lift, we need to use more singular kernel functions like those used in \cite{DT17}. Indeed, if
\begin{equation} \label{a34}
\frac{C_1}{r^{d+2 \alpha} } \leq  \omega(r) \leq \frac{C_2}{r^{d+2\alpha}}, \qquad \forall \  r \in (0,1),
\end{equation}
for some $\alpha \in (0,1)$ and constants $C_1$, $C_2>0$, then the corresponding $\lambda_\delta(\xx)$ grows as $|\xx|^{2\alpha}$ in high frequency part.      Consequently, we can get $2\alpha$ regularity pick-up,
\begin{equation} \label{a42}
\left \lVert \uu_{\delta} \right \rVert_{[H^{\alpha} (\Omega)]^d} +\left \lVert p_{\delta} \right \rVert_{L^2 (\Omega)} \leq C (\left \lVert \f \right \rVert_{[H^{-\alpha} (\Omega)]^d} + \left \lVert g \right \rVert_{L^2 (\Omega)}), 
\end{equation}
where the constant $C$ depends only on $C_1$. The proof of this result is in Appendix B.

Although the solution becomes more regular, the singular nature of the kernel function may complicate
the numerical integration in practice.  To avoid such issue,  a modified nonlocal system will be considered later
in section \ref{sec:reg}.

\section{Vanishing Nonlocality}
\label{sec:van}
Since the nonlocal model is constructed to approximate the original Stokes system.
After the well-posedness of the nonlocal model is obtained, one natural problem follows is the accuracy of the nonlocal approximation. In this section, we will show that the nonlocal solutions converge to its local counterpart as $\delta \to 0$ and the convergence rate is of second order.

\begin{theorem} \label{a39}
Assume that the kernels $\omega_{\delta}$, $\tilde{\omega}_{\delta}$ and $\hat{\omega}_{\delta}$ satisfy Assumption $\ref{a5}$. Let  $(\uu_{\delta}, p_{\delta})$ be the solution of nonlocal system $\eqref{a6}$, and $(\uu, p)$ be the solution of the following generalized Stokes system
\begin{equation} \label{a38}
\begin{cases}
-\Delta \textbf{u} + \nabla p= \textbf{f}, & \x \in \ \Omega, \\
\nabla \cdot \uu =g, & \x\in \ \Omega.
\end{cases}
\end{equation}
We have, for some constant 
 $C$ depend only on $m, M$ and independent of $\delta, \omega, \hat{\omega}$, $\textbf{f}$ and $g$, 
\begin{align} \label{a21}
&\left \lVert \uu-\uu_{\delta} \right \rVert_{[L^2(\Omega)]^d} \leq C ( \delta^2 \left \lVert \f \right \rVert_{[ L^2(\Omega)]^d} + \delta^{\min \{2, 1+\beta \} } \left \lVert g \right \rVert_{H^{\beta}(\Omega)}),\\
 \label{a22}
& \left \lVert p-p_{\delta} \right \rVert_{L^2(\Omega)} \leq C( \delta^{1+\eta } \left \lVert \f \right \rVert_{[H^{\eta}(\Omega)]^d} + \delta^{\beta} \left \lVert g \right \rVert_{H^{\beta}(\Omega)}),
\end{align}
for any $\eta \in [0,1]$, $\beta \in [0,2]$ and $\delta\in (0, 1)$.
If in addition,
\begin{equation} 
\frac{C_1}{r^{d+2 \alpha} } \leq  \omega(r) \leq \frac{C_2}{r^{d+2\alpha}}, \qquad \forall \  r \in (0,1),
\end{equation}
for some $\alpha \in (0,1)$ and constants $C_1$, $C_2>0$, then we have
\begin{align} \label{a46}
&\left \lVert \uu-\uu_{\delta} \right \rVert_{[H^{\alpha}(\Omega)]^d} \leq C ( \delta^{\min \{ 2, \ 2-2\alpha+\eta \} } \left \lVert \f \right \rVert_{[H^{\eta-\alpha}(\Omega)]^d} + \delta^{\min \{2, 1-\alpha+\beta \} } \left \lVert g \right \rVert_{H^{\beta}(\Omega)}),
\\  \label{a47}
&\left \lVert p-p_{\delta} \right \rVert_{L^2(\Omega)} \leq C( \delta^{1-\alpha+\eta } \left \lVert \f \right \rVert_{[H^{\eta-\alpha}(\Omega)]^d} + \delta^{ \beta } \left \lVert g \right \rVert_{H^{\beta}(\Omega)}) ,
\end{align}
for any $\eta\in [0,1]$ and  $\beta$ in $[0,2]$, where the constant $C$ depends only on $m, M$ and $C_1$.
\end{theorem}
\begin{proof}
Let us take the case $d=3$ for illustration. In the Fourier space, the difference between the nonlocal solution and the counterpart local solution can be explicitly expressed as:
\begin{align} 
\hat{\uu}(\xx)-\hat{\uu}_{\delta}(\xx) = & \big[ (\frac{1}{\lambda_{\delta}(\xx)}- \frac{1}{|\xx|^2} )I_d-( \frac{1}{\lambda_{\delta}(\xx)} \frac{\bb_{\delta}(\xx) \otimes \bb_{\delta} (\xx)}{
q_\delta(\xx)  } - \frac{ \xx \otimes \xx }{|\xx|^4 }) \big] \ \hat{\f}(\xx) \nonumber \\
& \qquad \qquad - i \big[ \frac{(\bb_{\delta}(\xx))}{
q_\delta(\xx) } - \frac{\xx}{|\xx|^2} \big]  \hat{g}(\xx) ,  \label{a28}\\
 \label{a29}
\hat{p}(\xx)-\hat{p}_{\delta}(\xx) = & -i \big[ \frac{(\bb_{\delta}(\xx))}{q_\delta(\xx) } - \frac{\xx}{|\xx|^2} \big]  \hat{\f}(\xx) + \big[ \frac{ \lambda_{\delta}(\xx)}{q_\delta(\xx) }  - 1 \big] \hat{g}(\xx).
\end{align}
Similar to that in the proof of the well-posedness theorem \ref{a37}, we also split $\xx$ to low frequency part ($0<\delta|\xx|<1$) and high frequency part ($\delta|\xx|\ge 1$).

For the low frequency part, $\delta |\xx|<1$, we obtain from Proposition \ref{fft-symbol}
\begin{equation} \label{a18}
\left| \frac{1}{\lambda_{\delta}(\xx)}- \frac{1}{|\xx|^2} \right| \leq \frac{1}{|\xx|^2- \frac{1}{20} |\xx|^2 \ ( \delta |\xx|)^2}-\frac{1}{|\xx|^2} \leq \frac{1}{10} \delta^2,
\end{equation}
for $C$ independent of $\delta$, $\xx$, $\omega$ and $\hat{\omega}$. Also using Proposition \ref{fft-symbol}, we have
\begin{equation}
\begin{split}
& 0\leq \frac{1}{| \bb_{\delta} (\xx) |^2 }-\frac{1}{q_\delta(\xx)   }
\leq  \frac{\delta^2 c_{\delta}(\xx) \lambda_{\delta}(\xx)  }{| \bb_{\delta} (\xx) |^4} \leq \frac{ \delta^2 (|\xx|^2) (M |\xx|^2) }{(\frac{1}{2} |\xx|)^4}=16 M  \delta^2,
\end{split}
\end{equation}
and
\begin{align*}
\left| \frac{1}{\lambda_{\delta}(\xx)} \frac{\bb_{\delta}(\xx) \otimes \bb_{\delta} (\xx)}{q_\delta(\xx)   } - \frac{ \xx \otimes \xx }{|\xx|^4 } \right| \leq & \left| \frac{1}{\lambda_{\delta}(\xx)} \frac{\bb_{\delta}(\xx) \otimes \bb_{\delta} (\xx)}{| \bb_{\delta} (\xx) |^2 } - \frac{1}{|\xx|^2} \frac{ \xx \otimes \xx }{|\xx|^2 } \right| \\
&+  \frac{1}{\lambda_{\delta}(\xx)} \ \left| \frac{\bb_{\delta}(\xx) \otimes \bb_{\delta} (\xx)}{q_\delta(\xx)   } -  \frac{\bb_{\delta}(\xx) \otimes \bb_{\delta} (\xx)}{| \bb_{\delta} (\xx) |^2  } \right|.
\end{align*}
By the fact that $\bb_{\delta} (\xx)=|\bb_{\delta} (\xx)|\xx/|\xx|$, we get
\begin{equation}  \label{a19}
\begin{split}
\left| \frac{1}{\lambda_{\delta}(\xx)} \frac{\bb_{\delta}(\xx) \otimes \bb_{\delta} (\xx)}{q_\delta(\xx)   } - \frac{ \xx \otimes \xx }{|\xx|^4 } \right| & \leq  \left|\frac{1}{\lambda_{\delta}(\xx)}- \frac{1}{|\xx|^2}\right| 
+ \frac{ \delta^2 c_{\delta}(\xx)}{q_\delta(\xx)  }  \\
& < (1+4M) \delta^2, 
\end{split}
\end{equation}
Moreover, we can get
\begin{equation} \label{a20}
\begin{split}
\left| \frac{(\bb_{\delta}(\xx))^T}{q_\delta(\xx) } - \frac{\xx^T}{|\xx|^2} \right| & \leq 
\left| \frac{\bb_{\delta}(\xx) }{q_\delta(\xx)  }- \frac{\bb_{\delta}(\xx) }{| \bb_{\delta} (\xx) |^2} \right| 
 + \left|\frac{\bb_{\delta}(\xx) }{| \bb_{\delta} (\xx) |^2  }- \frac{\xx^T}{|\xx|^2}\right| \\
&\leq \frac{\delta^2 c_{\delta}(\xx) \lambda_{\delta}(\xx)}{|\bb_{\delta}(\xx)|^3}+\left|\frac{1}{|\bb_{\delta}(\xx)|}-\frac{1}{|\xx|}\right| \\ 
&\leq \frac{\delta^2 (|\xx|^2) (M |\xx|^2)}{(\frac{1}{2}|\xx|)^3} + \frac{1}{5} \delta^2 |\xx| \leq (8M+1) \delta^2 |\xx|.
\end{split}
\end{equation}

Now we plug the estimates $\eqref{a18}$, $\eqref{a19}$ and $\eqref{a20}$ into equation $\eqref{a28}$,
\begin{equation}\label{u-vanish-1}
\begin{split}
| \hat{\uu}(\xx)-\hat{\uu}_{\delta}(\xx) | \leq& (2+4M) \ \delta^2  \ \hat{\f}(\xx) + (8M+1) \ \delta^2 \ |\xx| \hat{g}(\xx) \\
 \leq &(2+4M) \ \delta^2  \ \hat{\f}(\xx) + (8M+1) \ \delta^{1+\eta} \ |\xx|^{\eta} \hat{g}(\xx),
\end{split}
\end{equation}
for any $\eta \in [0, 1]$.

To estimate the difference of the pressure, again from the Proposition \ref{fft-symbol},
\begin{equation} \label{a35}
\begin{split}
\left| \frac{ \lambda_{\delta}(\xx) }{q_\delta(\xx) }  - 1 \right| 
& \leq \left| \frac{ \lambda_{\delta}(\xx) - | \bb_{\delta} (\xx) |^2 - \delta^2 c_{\delta}(\xx) \lambda_{\delta}(\xx)   }{| \bb_{\delta} (\xx) |^2 }   \right|  \\
& \leq  \left| \frac{ \lambda_{\delta}(\xx) - | \bb_{\delta} (\xx) |^2   }{| \bb_{\delta} (\xx) |^2 }   \right|
 +\frac{\delta^2 c_{\delta}(\xx) \lambda_{\delta}(\xx) }{| \bb_{\delta} (\xx) |^2 }\\
 & \leq \frac{|\xx|^2-(|\xx|-\frac{\delta^2 |\xx|^2}{10}|\xx|)^2}{(\frac{1}{2}|\xx|)^2} + 
 \frac{\delta^2 M |\xx|^2 |\xx|^2}{(\frac{1}{2}|\xx|)^2}  \leq (4M+1) \delta^2 |\xx|^2,
\end{split}
\end{equation}
Together with $\eqref{a20}$, we get for any $\eta \in [0, 1]$ and  $\beta \in [0, 2]$, 
\begin{equation} \label{a43}
\begin{split}
| \hat{p}(\xx)-\hat{p}_{\delta}(\xx) | \leq (8M+1) \ \delta^2 \ |\xx| \ \hat{\f}(\xx) + (4M+1) \delta^2 |\xx|^2 \ \hat{g}(\xx) \\
\leq (8M+1) \ \delta^{1+\eta} \ |\xx|^{\eta} \ \hat{\f}(\xx) +(4M+1) \ \delta^{\beta} \ |\xx|^{\beta} \hat{g}(\xx) .
\end{split}
\end{equation}

 For the high frequency part, $\delta |\xx| \geq 1$, we have 
\begin{equation} \label{a30}
\begin{split}
\left| (\frac{1}{\lambda_{\delta}(\xx)}- \frac{1}{|\xx|^2}) I_d - ( \frac{1}{\lambda_{\delta}(\xx)} \frac{\bb_{\delta}(\xx) \otimes \bb_{\delta} (\xx)}{q_\delta(\xx)  }  - \frac{1}{|\xx|^2} \frac{ \xx \otimes \xx }{|\xx|^2 }) \right| \\
\leq 2 \left(\frac{1}{\lambda_{\delta}(\xx)} + \frac{1}{|\xx|^2}\right) < 32 \delta^2+2 \delta^2=34 \delta^2,
\end{split}
\end{equation}
and
\begin{equation} \label{a31}
\left| \frac{(\bb_{\delta}(\xx))^T}{q_\delta(\xx)  } - \frac{\xx^T}{|\xx|^2} \right| \leq  \frac{1}{2 \delta  \sqrt{\lambda_{\delta}(\xx)  c_{\delta}(\xx)} } + \frac{1}{|\xx|} <(\frac{8}{\sqrt{m}}+1) \delta;
\end{equation}
Now we plug $\eqref{a30}$ and $\eqref{a31}$ into equation $\eqref{a28}$, we get for $\eta\in [0,1]$, 
\begin{equation}\label{u-vanish-2}
\begin{split}
| \hat{\uu}(\xx)-\hat{\uu}_{\delta}(\xx) | &  \leq 34 \ \delta^2  \ \hat{\f}(\xx) \ + \ (\frac{8}{\sqrt{m}}+1) \ \delta \ \hat{g}(\xx) \\
 &\leq\, 34 \ \delta^2  \ \hat{\f}(\xx) \ + \ (\frac{8}{\sqrt{m}}+1) \ \delta^{1+\eta} \ |\xx|^{\eta} \ \hat{g}(\xx), \
 \end{split}
\end{equation}
Combining \eqref{u-vanish-1} and \eqref{u-vanish-2}, \eqref{a21} is proved.
Moreover, we have
\begin{equation} \label{a36}
| \frac{ \lambda_{\delta}(\xx)}{q_\delta(\xx) }  - 1 | \leq 1+\frac{16}{m},
\end{equation}
Substituting $\eqref{a36}$ and $\eqref{a31}$ into $\eqref{a29}$, we have
\begin{equation} \label{a44}
\begin{split}
| \hat{p}(\xx)-\hat{p}_{\delta}(\xx) | & \leq (\frac{8}{\sqrt{m}}+1) \delta \ \hat{\f}(\xx) + (\frac{16}{m}+1) \ \hat{g}(\xx)  \\
& \leq  (\frac{8}{\sqrt{m}}+1) \ \delta^{1+\eta} \ |\xx|^{\eta}  \ \hat{\f}(\xx) \ + \ \frac{16}{m} \delta^{\beta} \ |\xx|^{\beta}  \hat{g}(\xx), 
\end{split}
\end{equation}
 for any $\eta \in [0, 1]$ and  $\beta \in [0, 2]$. Hence, 
$\eqref{a22}$ follows from \eqref{a43} and \eqref{a44}.

The second part can be proved similarly. The details of the proof can be found in Appendix C.  \end{proof}

\begin{remark}
\begin{enumerate}
\item
The case $g \equiv 0$ in the system $\eqref{a38}$ gives the standard Stokes system. Under such case, we end up with
\begin{equation} 
\left \lVert \uu-\uu_{\delta} \right \rVert_{[L^2(\Omega)]^d} \leq C  \delta^2 \left \lVert \f \right \rVert_{[L^2(\Omega)]^d} ,
\end{equation}
and
\begin{equation} 
\left \lVert p-p_{\delta} \right \rVert_{L^2(\Omega)} \leq C \delta^{\min \{2, 1+\eta \} } \left \lVert \f \right \rVert_{[H^{\eta}(\Omega)]^d} ,
\end{equation}
for any $\eta \in [0,1]$ and some constant $C$ independent of $\delta, \omega, \hat{\omega}$ and $\tilde{\omega}$.

\item
When $g$ is only an $L^2$ function and has no higher regularity, we cannot directly obtain from $\eqref{a22}$ that $p_{\delta}$ converges to $p$ in $L^2$ as $\delta \to 0$.  However, such convergence still holds, but the rate depends on $g$.  To see why this is the case,  we first obtain from the above calculation that $ \big[ \frac{ \lambda_{\delta}(\xx)}{q_\delta(\xx) }  - 1 \big] $ is bounded above by some constant, and it goes to $0$ as $\delta |\xx| \to 0$. Then for any $\epsilon>0$, we can find $N \in \mathbb{N}^+$ that $\sum \limits_{\xx=N}^{\infty} \hat{g}(\xx) < \epsilon$; also, we can find $\delta_0$ such that $\big[ \frac{ \lambda_{\delta}(\xx)}{q_\delta(\xx) }  - 1 \big] < \epsilon$  for all $\delta |\xx| < \delta_0 \ N $. Hence for any $\delta<\delta_0$, we have
\begin{align*}
 \sum \limits_{\xx=1}^{\infty} \big[ \frac{ \lambda_{\delta}(\xx)}{q_\delta(\xx) }  - 1 \big] \hat{g}(\xx) &= 
  \sum \limits_{\xx=1}^{N-1} \big[ \frac{ \lambda_{\delta}(\xx)}{q_\delta(\xx) }  - 1 \big] \hat{g}(\xx) \
  +\sum \limits_{\xx=N}^{\infty} \big[ \frac{ \lambda_{\delta}(\xx)}{q_\delta(\xx) }  - 1 \big] \hat{g}(\xx) \\
& < \epsilon \left \lVert g \right \rVert_{L^2} +C \epsilon.
\end{align*}
since $\epsilon>0$ is arbitrary, we see that 
$$ \sum \limits_{\xx=1}^{\infty} \big[ \frac{ \lambda_{\delta}(\xx)}{q_\delta(\xx) }  - 1 \big] \hat{g}(\xx) \to 0,\; \mbox{ as }\; 
\delta \to 0.$$
 Therefore, we can easily observe from $\eqref{a29}$ that $p-p_{\delta} \to 0$ in $L^2$ as $\delta \to 0$.

\end{enumerate}
\end{remark}

\section{A Modified Nonlocal Model}
\label{sec:reg}
We see from the estimates \eqref{a21}-\eqref{a22} in the
theorem \ref{a39} that the second order convergence to the local limit rests upon the regularity of data.
The latter can be improved by considering a modified nonlocal model with regularized right hand sides
through a convolution operation. 
If the kernel function used in the convolution is in $H^\gamma$, then the regularity of the right hand side is
 lifted by order $\gamma$. Consequently, the 
solution of the nonlocal system also becomes more regular, and the convergence order gets higher as well. 

Based on this observation, we modify the nonlocal model \eqref{a6} as follows:
\begin{equation} \label{a41}
\begin{cases}
-\mathcal{L}_{\delta} \textbf{u}^0_{\delta}(\x) + \mathcal{G}_{\delta}  p^0_{\delta} (\x)= \textbf{f}_{\delta} (\x), & \x \in \Omega, \\
\mathcal{D}_{\delta} \uu^0_{\delta} (\x) - \delta^2 \mathcal{R}_{\delta} p^0_{\delta}(\x) =  g_{\delta}(\x), & \x \in \Omega,
\end{cases}
\end{equation}
where
\begin{equation}
\f_{\delta}(\x)=\int_{\mathbb{R}^d} \bar{\omega}_{\delta} (\x-\y) \ \f(\y) d\y, \qquad g_{\delta}(\x)=\int_{\mathbb{R}^d} \bar{\omega}_{\delta} (\x-\y) \ g(\y) d\y,
\end{equation}
here the kernel $\bar{\omega}_{\delta}$ is rescaled from a function $\bar{\omega}$ and satisfies the following assumption.

\begin{assumption} \label{b1}
\begin{enumerate}
\item $\bar{\omega}$ is non-negative, radial symmetric and is supported in the unit sphere.
\item $\bar{\omega}$ is  integrable in ${\mathbb{R}^d} $ and 
\begin{equation}
 \int_{\mathbb{R}^d} \bar{\omega} (|\x|) d\x=1,
\end{equation}
\item  $\bar{\omega}$ lies in $H^{\gamma}(\mathbb{R}^d),\; \gamma>0$.
\item $\displaystyle
\bar{\omega}_{\delta} (|\x|)=\frac{1}{\delta^d} \bar{\omega} (\frac{|\x|}{\delta}).$
 \end{enumerate}
 \end{assumption}
 
Based on the assumption \ref{b1}, we have that for any $s \in \mathbb{R}$,
\begin{align}
  \label{eq:smooth-f}
  \|f_\delta\|_{[H^{\gamma+s} (\Omega)]^d}\le C_\delta \|f\|_{[H^{s} (\Omega)]^d},\quad \|g_\delta\|_{[H^{\gamma+s} (\Omega)]^d}\le C_\delta \|g\|_{[H^{s} (\Omega)]^d},
\end{align}
where $C_\delta$ is a constant depend on $\delta$ and $m$. Then following the same proof as that in Theorem \ref{a37}, we have that  
\begin{equation} \label{a48}
\left \lVert \uu^0_{\delta} \right \rVert_{[H^{\gamma} (\Omega)]^d} \leq C_\delta (\left \lVert \f \right \rVert_{[L^2 (\Omega)]^d} + \left \lVert g \right \rVert_{L^2 (\Omega)}), 
\end{equation}
and
\begin{equation}  \label{a50}
\left \lVert p^0_{\delta} \right \rVert_{L^2 (\Omega)} \leq C_\delta (\left \lVert \f \right \rVert_{[H^{-\gamma} (\Omega)]^d} + \left \lVert g \right \rVert_{L^2 (\Omega)}),
\end{equation}
where $C_\delta$ depend on $m$ and $\delta$. Moreover, the modified nonlocal model still has the same order of convergence as $\uu$ and $p$ in \eqref{a21} and \eqref{a22}. 
\begin{theorem} \label{a49}
Assume that the kernels $\omega_{\delta}$, $\tilde{\omega}_{\delta}$ and $\hat{\omega}_{\delta}$ satisfy Assumption $\ref{a5}$, and $\bar{\omega}_{\delta}$ satisfy the assumption \ref{b1}. Let  $(\uu^0_{\delta}, p^0_{\delta})$ be the solution of the modified nonlocal system $\eqref{a41}$, and $(\uu, p)$ be the solution of the generalized Stokes system $\eqref{a38}$, then
\begin{equation} \label{a55}
\left \lVert \uu-\uu^0_{\delta} \right \rVert_{[H^{\gamma}(\Omega)]^d} \leq C ( \delta^{ \min \{ 2, 2-\gamma+\eta \} } \left \lVert \f \right \rVert_{[H^{\eta}(\Omega)]^d} + \delta^{\min \{2, 1+\beta-\gamma \} } \left \lVert g \right \rVert_{H^{\beta}(\Omega)}),
\end{equation}
and
\begin{equation} \label{a56}
\left \lVert p-p^0_{\delta} \right \rVert_{H^{\gamma}(\Omega)} \leq C( \delta^{\min \{2, 1-\gamma+ \eta \} } \left \lVert \f \right \rVert_{[H^{\eta}(\Omega)]^d} + \delta^{\min \{ 2, \beta-\gamma \} } \left \lVert g \right \rVert_{H^{\beta}(\Omega)}) ,
\end{equation}
for any $\eta>0$ and $\beta>\gamma $, where the constant $C$ depend only on the constant $m$ and $M$.
\end{theorem}

\begin{proof}
First, we split the error to two terms:
\begin{align*}
  \uu-\uu_\delta^0=&(\uu-\bar{\uu}_\delta)+(\bar{\uu}_\delta-\uu_\delta^0),\\
p-p_\delta^0=&(p-\bar{p}_\delta)+(\bar{p}_\delta-p_\delta^0),
\end{align*}
where $\bar{\uu}_\delta$ and $\bar{p}_\delta$ solve a local Stokes system but with regularized data:
\begin{align*}
- \Delta \bar{\uu}_\delta + \nabla \bar{p}_\delta=& \textbf{f}_\delta,\quad \nabla \cdot \bar{\uu}_\delta =g_\delta.
\end{align*}
For the second term, $\bar{\uu}_\delta-\uu_\delta^0$ and $\bar{p}_\delta-p_\delta^0$,  
by mimicking the proof of theorem $\ref{a39}$ and using the regularity lift result \eqref{eq:smooth-f}, we can conclude the convegence estimate,
\begin{equation} 
\left \lVert \bar{\uu}_\delta-\uu^0_{\delta} \right \rVert_{[H^{\gamma}(\Omega)]^d} \leq C ( \delta^{ \min \{ 2, 2-\gamma+\eta \} } \left \lVert \f \right \rVert_{[H^{\eta}(\Omega)]^d} + \delta^{\min \{2, 1+\beta-\gamma \} } \left \lVert g \right \rVert_{H^{\beta}(\Omega)}),\nonumber
\end{equation}
and
\begin{equation} 
\left \lVert \bar{p}_\delta-p^0_{\delta} \right \rVert_{H^{\gamma}(\Omega)} \leq C( \delta^{\min \{2, 1-\gamma+ \eta \} } \left \lVert \f \right \rVert_{[H^{\eta}(\Omega)]^d} + \delta^{\min \{ 2, \beta-\gamma \} } \left \lVert g \right \rVert_{H^{\beta}(\Omega)}).\nonumber
\end{equation}
The calculation of the first term is also straightforward. Using Fourier transform, we have for $\xx\neq 0$, 
\begin{equation} \label{a45}
\begin{split}
\left( \ \begin{matrix}
\D \hat{\uu}(\xx)-\hat{\bar{\uu}}_{\delta}(\xx)  \\ \D \hat{p}(\xx)-\hat{\bar{p}}_{\delta}(\xx)
\end{matrix} \ \right) 
= \left( \ \begin{matrix}
\D \frac{1}{|\xx|^2}(I_d-\frac{\xx \ \otimes \ \xx}{| \xx |^2  } \ ) & 
\D -i \frac{\xx}{|\xx|^2} \\ \D  -i \frac{\xx^T}{|\xx|^2}  &\D 1
\end{matrix} \ \right) \left( \ \begin{matrix}
\D \hat{\f}(\xx)-\hat{\f}_{\delta}(\xx)  \\ \D \hat{g}(\xx)-\hat{g}_{\delta}(\xx)
\end{matrix} \ \right),
\end{split}
\end{equation}
which implies that
\begin{align}
\label{eq:err-delta-u}
  |\hat{\uu}(\xx)-\hat{\bar{\uu}}_{\delta}(\xx)|\le& |\xx|^{-2}|\hat{\f}(\xx)-\hat{\f}_{\delta}(\xx) |+|\xx|^{-1}|\hat{g}(\xx)-\hat{g}_{\delta}(\xx)|,\\
\label{eq:err-delta-p}
|\hat{p}(\xx)-\hat{\bar{p}}_{\delta}(\xx)|\le& |\xx|^{-1}|\hat{\f}(\xx)-\hat{\f}_{\delta}(\xx) |+|\hat{g}(\xx)-\hat{g}_{\delta}(\xx)|
\end{align}
Using the definition of $\mathbf{f}_\delta$ and $g_\delta$, we have
\begin{equation}
\begin{cases}
\hat{\f}(\xx) - \hat{\f}_{\delta}(\xx) = d_{\delta}(\xx) \ \hat{\f}(\xx), \\
 \hat{g}(\xx) - \hat{g}_{\delta}(\xx)= d_{\delta}(\xx) \ \hat{g}(\xx),
\end{cases}
\end{equation}
where
\begin{align*}
d_{\delta} (\xx)=&\int_{|\s| \leq \delta} \bar{\omega}_{\delta}(|\s|)(1-\cos( \xx \cdot \s)) d \s \\
=&4 \pi \int_0^{\pi/2} \sin (\phi) \int_0^1 r^2 \bar{\omega}(r) (1-\cos(r \ \cos(\phi) \delta |\xx|)) dr d\phi.
\end{align*}
For $\delta |\xx|<1$, we see
\begin{align} \label{a51}
|d_{\delta} (\xx)|=&4 \pi \int_0^{\pi/2} \sin (\phi) \int_0^1 r^2 \bar{\omega}(r) (1-\cos(r \ \cos(\phi) \delta |\xx|)) dr d\phi \nonumber\\
\leq& 4 \pi \int_0^{\pi/2} \sin (\phi) \int_0^1 r^2 \bar{\omega}(r) \big( \frac{1}{2} r^2 \cos^2(\phi) \delta^2 |\xx|^2 \big) dr d\phi \nonumber\\
\leq & \delta^2 |\xx|^2 \int_0^{\pi/2} \frac{1}{2}  \cos^2(\phi) \sin (\phi) d\phi  \int_0^1 4 \pi r^2 \bar{\omega}(r)  dr =\frac{1}{6} \delta^2 |\xx|^2.
\end{align}
while for $\delta |\xx| \geq 1$, 
\begin{equation} \label{a52}
\begin{split}
|d_{\delta} (\xx)| \le 2.
\end{split}
\end{equation}
Using these estimate of $d_\delta(\xx)$ and \eqref{eq:err-delta-u}, \eqref{eq:err-delta-p}, it is easy to show that 
\begin{equation} 
\left \lVert \uu-\bar{\uu}_\delta \right \rVert_{[H^{\gamma}(\Omega)]^d} \leq C ( \delta^{ \min \{ 2, 2-\gamma+\eta \} } \left \lVert \f \right \rVert_{[H^{\eta}(\Omega)]^d} + \delta^{\min \{2, 1+\beta-\gamma \} } \left \lVert g \right \rVert_{H^{\beta}(\Omega)}),\nonumber
\end{equation}
and
\begin{equation} 
\left \lVert p-\bar{p}_\delta \right \rVert_{H^{\gamma}(\Omega)} \leq C( \delta^{\min \{2, 1-\gamma+ \eta \} } \left \lVert \f \right \rVert_{[H^{\eta}(\Omega)]^d} + \delta^{\min \{ 2, \beta-\gamma \} } \left \lVert g \right \rVert_{H^{\beta}(\Omega)}),\nonumber
\end{equation}
which complete the proof.
\end{proof}

\section{Numerical discretization}
\label{sec:num}

Given the well-posedness of the nonlocal system \eqref{a6} with periodic boundary condition, we may think about its numerical simulation as in the case of \cite{DT17, DY5}. According to our analysis in the previous sections, the Fourier spectral method can be considered for numerical approximation. The convergence of such method can be obtained from classical Fourier analysis.

Let $(\uu_{\delta}^N, p_{\delta}^N)$ be Fourier spectral approximation of $(\uu_{\delta}, p_{\delta})$. In other words,
\begin{equation}
\uu_\delta^N(\x)=\frac{1}{(2\pi)^d} \sum_{\xx \in \Z^d,  0 < |\xi_i| \leq N} \hat{\uu}_\delta (\xx)e^{i \xx \cdot \x},
\quad
p_\delta^N(\x)=\frac{1}{(2\pi)^d} \sum_{\xx \in \Z^d, 0 < |\xi_i| \leq N} \hat{p}_\delta (\xx)e^{i \xx \cdot \x},\nonumber
\end{equation}
Here $\xx=(\xi_1, \xi_2)^T$ for $d=2$ and $\xx=(\xi_1, \xi_2, \xi_3)^T$ for $d=3$.
We can easily observe that $(\uu_{\delta}^N, p_{\delta}^N)$ are the projection of $(\uu_{\delta}, p_{\delta})$ over all Fourier modes with wave number no larger than $N$. Under this eyesight, for a fixed $\delta>0$ as $N \to \infty$, we get the following convergence result for $\f \in [L^2(\Omega)]^d, g \in L^2(\Omega)$ and also for smoother data.



\begin{theorem} \label{a32}
Let $(\uu_{\delta}^N(\x), p_{\delta}^N(\x))$ be Fourier spectral approximation to the $\eqref{a6}$. We assume also that the Assumption $\ref{a5}$ hold true for the kernels. Then for $\f \in [L^2(\Omega)]^d$, and $g \in L^2(\Omega)$ we have as $N \to \infty$, 
\begin{equation}
\left \lVert \uu_{\delta}^N-\uu_{\delta} \right \rVert_{[L^2(\Omega)]^d} \to 0 \qquad \mbox{and} \qquad  \left \lVert p_{\delta}^N-p_{\delta} \right \rVert_{L^2(\Omega)} \to 0.
\end{equation}
Moreover, if $\omega(r)$ satisfy $\eqref{a34}$ for some $\alpha \in (0,1)$ and constants $C_1, \ C_2 \in \mathbb{R}^+$, then for any $s \geq 0$ and $\gamma \geq 0$, with $C$ independent of $\delta$, $N$, $\f$, $g$ and $s$, we have as $N \to \infty$,
\begin{equation}
\left \lVert \uu_{\delta}^N -\uu_{\delta} \right \rVert_{[H^{\gamma+\alpha}(\Omega)]^d} + \left \lVert p_{\delta}^N -p_{\delta} \right \rVert_{H^{\gamma}(\Omega)} \leq \frac{C}{N^s} (\left \lVert  \partial^s \f \right \rVert_{[H^{\gamma-\alpha}(\Omega)]^d} + \left \lVert  \partial^s g \right \rVert_{H^{\gamma}(\Omega)}),
\end{equation}
where $s>0$ denote the total order of differentiation of a partial differential operator $\partial^s$.
\end{theorem}

\begin{theorem} \label{a33}
Let $(\uu_{\delta}^N(\x), p_{\delta}^N(\x))$ and $(\uu^N(\x), p^N(\x))$ be the Fourier spectral approximations to the solutions of $\eqref{a6}$ and $\eqref{a38}$ respectively. We assume also that Assumption $\ref{a5}$ hold true for the kernels. Then with $C$ independent of $\delta$, $N$, $\f$ and $g$, we have
\begin{equation}
\left \lVert \uu_{\delta}^N-\uu^N \right \rVert_{[L^2(\Omega)]^d} \leq C (\delta^2 \left \lVert \f^N \right \rVert_{[L^2(\Omega)]^d} +  \delta^{\min \{2, 1+\beta\} } \left \lVert g^N \right \rVert_{H^{\beta}(\Omega)} ),
\end{equation}
and
\begin{equation}
\left \lVert p_{\delta}^N-p^N \right \rVert_{L^2(\Omega)} \leq C( \delta^{\min \{2, 1+\eta\} } \left \lVert \f^N \right \rVert_{[H^{\eta}(\Omega)]^d} + \delta^{\beta }  \left \lVert g^N \right \rVert_{H^{\beta}(\Omega)}),
\end{equation}
for $\eta \in [0,1]$ and $\beta \in [0,2]$.
\end{theorem}

The above Theorems $\ref{a32}$ and $\ref{a33}$ are immediately followed from Theorems $10$ and $11$ in \cite{DT17}.


The above two theorems imply that such Fourier spectral approximation of \eqref{a6} satisfies the so called asymptotic compatibility (AC) condition, which is a nice property for the numerical discretization to nonlocal models.
The AC property is visually illustrated in Fig. \ref{fig:diagram} (cf. definitions and 
theory initially developed in \cite{TD14}).  It describes the numerical discretization of the local Stokes equation  via nonlocal integral approximation \cite{DT17}.

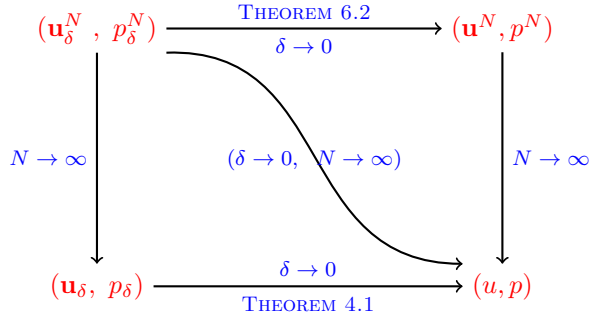
\begin{figure}[t]
  \centering
       \begin{tikzpicture}[scale=0.98]
     \tikzset{to/.style={->,line width=.8pt}}   
 \node(v1) at (0,3.5) {\textcolor{red}{$(\uu_{\delta}^N \ , \ p_{\delta}^N)$}};
  \node (v2) at (5.5,3.5) {\textcolor{red}{$(\uu^N, p^N)$}};
   \node (v3) at (0,0) {\textcolor{red}{$(\uu_{\delta}, \ p_{\delta})$}};
    \node (v4) at (5.5,0) {\textcolor{red}{$(u,p)$}};
     \draw[to] (v1.east) -- node[midway,above] {\footnotesize{\textcolor{blue}{\sc  Theorem \ref{a33}}}}  node[midway,below] {\footnotesize{\textcolor{blue}{$\delta\to0$}}}      
     (v2.west);
     \draw[to] (v1.south) -- node[midway,left] {\footnotesize{\textcolor{blue}{\sc $N \to \infty$}}}  (v3.north);
       \draw[to] (v3.east) -- node[midway,below] {\footnotesize{\textcolor{blue}{\sc Theorem \ref{a39}}}} node[midway,above] {\footnotesize{\textcolor{blue}{$\delta\to0$}}} (v4.west);
       \draw[to] (v2.south) -- node[midway,right] {\footnotesize{\textcolor{blue}{\sc $N \to \infty$}}} (v4.north);
     \draw[to] (v1.south east) to[out = 2, in = 180, looseness = 1.2] node[midway] {\footnotesize\hbox{\shortstack[l]{ {\textcolor{blue}{\sc $\;\;$}}\\ {\textcolor{blue}{($\delta\to0,\; \ N \to \infty$)}} }}} (v4.north west);
      \end{tikzpicture}
   \caption{\footnotesize A diagram for asymptotically compatible schemes and convergence results.}\label{fig:diagram}
\end{figure}

\section{Conclusion and discussions}
\label{sec:con}

In this paper, we propose a nonlocal approximation to the linear Stokes system. Under more general assumptions on the kernel functions (than those studied before, e.g. \cite{DT17}), we prove that the nonlocal model is well-posed and converges to the original Stokes system as the length of the nonlocal interaction goes to zero. 
A key ingredient in achieving this is to add a relaxation term to the incompressibility constraint, which is reminiscent to
commonly used techniques of using small artificial compressibility to approximate the incompressibility.
Furthermore, we can show that the rate of convergence is of second order for regular $\textbf{f}$ and $g$. In addition, we also give more precise upper bounds of the constants in the estimate of the convergence rate and such bounds are independent on the kernel functions.
Comparing to the previous works presented in \cite{DT17,hwi19}, our model allows the use of radial symmetric nonlocal interactions and has more freedom in choosing the nonlocal kernel functions. This offers more
convenience in the practical  implementation of numerical schemes based on the nonlocal approximations
(such as the SPH used in the fluid simulations).

Since Fourier transform is used as the main tool for the mathematical analysis, the techniques in this paper is 
more directly applicable to the case of periodic boundary conditions. An interesting problem is to generalize the nonlocal model to more general geometric domains with general boundary conditions, for instance, 
 the nonlocal Dirichlet boundary condition. Extensions to nonlinear problems are also of much practical importance.
 Although new techniques need to be developed for these further studies, 
the results obtained in this paper give some valuable insights to the future development of nonlocal modeling and nonlocal
relaxations of local continuum models.





\input{appendix}

\bibliographystyle{abbrv}
\bibliography{reference}

\end{document}

%% file: appendix.tex
\appendix

\section{Proof of Proposition \ref{fft-symbol}}
First, the Fourier symbols $\lambda_{\delta} (\xx)$, $\bb_{\delta}( \xx)$ and $c_\delta(\xx)$ can be expressed by
\begin{equation} \label{a2}
\lambda_{\delta} (\xx)=
\begin{cases}
\D
4 \int_0^{\pi/2} \int_0^{\delta} r \omega_{\delta}(r) (1-\cos(r \ \cos(\phi) |\xx|))dr d\phi  \qquad \mbox{for} \ d=2, \\
\D 4 \pi \int_0^{\pi/2} \sin (\phi) \int_0^{\delta} r^2 \omega_{\delta}(r) (1-\cos(r \ \cos(\phi) |\xx|))dr d\phi  \qquad \mbox{for} \ d=3, 
\end{cases}
\end{equation}
and
\begin{equation}
\bb_{\delta} (\xx)=b_{\delta}(|\xx|)\frac{\xx}{|\xx|},\nonumber
\end{equation}
where the scalar coefficient $b_{\delta}(|\xx|)$ is given by
\begin{equation}
b_{\delta} (\xx)=
\begin{cases}
\D
4 \int_0^{\pi/2} \cos(\phi) \ \int_0^{\delta} r \hat{\omega}_{\delta}(r) \ \sin(r \ \cos(\phi) |\xx|) \ dr d\phi  \qquad \mbox{for} \ d=2, \\
\D 4 \pi \int_0^{\pi/2} \cos(\phi) \sin (\phi) \ \int_0^{\delta} r^2 \hat{\omega}_{\delta}(r) \ \sin(r \ \cos(\phi) |\xx|) \ dr d\phi  \qquad \mbox{for} \ d=3,
\end{cases}\nonumber
\end{equation}
also
\begin{equation} 
c_{\delta} (\xx)=
\begin{cases}
\D
4 \int_0^{\pi/2} \int_0^{\delta} r \tilde{\omega}_{\delta}(r) (1-\cos(r \ \cos(\phi) |\xx|))dr d\phi  \qquad \mbox{for} \ d=2, \\
\D
4 \pi \int_0^{\pi/2} \sin (\phi) \int_0^{\delta} r^2 \tilde{\omega}_{\delta}(r) (1-\cos(r \ \cos(\phi) |\xx|))dr d\phi  \qquad \mbox{for} \ d=3.
\end{cases}\nonumber
\end{equation}

Here, we only address the case $d=3$. Recall that 
\begin{equation}
\begin{split}
\lambda_{\delta}(\xx)=4 \pi \int_0^{\pi/2} \sin (\phi) \int_0^{\delta} r^2 \omega_{\delta}(r) (1-\cos(r \ \cos(\phi)  |\xx|)) dr d\phi  \\
= \frac{4 \pi}{\delta^2} \int_0^{\pi/2} \sin (\phi) \int_0^1 r^2 \omega(r) (1-\cos(r \ \cos(\phi) \delta |\xx|)) dr d\phi.
\end{split}
\end{equation}
We now consider different cases. First, 
 for $\delta |\xx|>1$, we have 
\begin{equation}
\begin{split}
\frac{\delta^2}{4 \pi} \lambda_{\delta}(\xx)= \int_0^{\pi/2} \sin (\phi) \int_0^1 r^2 \omega(r) (1-\cos(r \ \cos(\phi) \delta |\xx|)) dr d\phi \\
\geq \frac{\sqrt{3}}{2} \int_{\pi/3}^{\pi/2}  \int_0^1 r^2 \omega(r) (1-\cos(r \ \cos(\phi) \delta |\xx|)) dr d\phi \\
= \frac{\sqrt{3}}{2} \int_0^{1/2} \frac{1}{\sqrt{1-t^2}} \int_0^1 r^2 \omega(r) (1-\cos(r \ t \delta |\xx|)) dr dt \\
\geq  \frac{\sqrt{3}}{2} \int_0^{1/2}  \int_0^1 r^2 \omega(r) (1-\cos(r \ t \delta |\xx|)) dr dt \\
=\frac{\sqrt{3}}{2} \int_0^1 r^4 \omega(r)  \int_0^{1/2} \frac{1-\cos(rt \delta |\xx|)}{r^2} dt  dr \\
=\frac{\sqrt{3}}{2} \int_0^1 r^4 \omega(r)  \frac{r \delta | \xx | /2- \sin (r \delta | \xx | /2)}{r^3 \delta |\xx|} dr.
\end{split}
\end{equation}

We again consider first the case
$0<r< \frac{1}{\delta |\xx|}$, then $r \delta |\xx|<1$, so
\begin{equation}
\frac{r \delta | \xx | /2- \sin (r \delta | \xx | /2)}{r^3 \delta |\xx|} >  \frac{r \delta | \xx | /2- \sin (r \delta | \xx | /2)}{8 \ (r \delta |\xx|/2)^3} > \frac{1}{2} \cdot  \frac{1}{6} \cdot \frac{1}{8}=\frac{1}{96}.
\end{equation}

Then, if $\frac{1}{\delta |\xx|} \leq r \leq 1$, then $r \delta |\xx|>1$, and
\begin{equation}
\frac{r \delta | \xx | /2- \sin (r \delta | \xx | /2)}{r^3 \delta |\xx|} > \frac{\frac{1}{25} \ r \delta |\xx|/2}{r^3 \delta |\xx|}=\frac{1}{50r^2} \geq \frac{1}{50}.
\end{equation}

Hence we can make the conclusion that  for $\delta |\xx| \geq 1$,
\begin{equation} \label{a13}
\lambda_{\delta}(\xx) \geq \frac{4 \pi}{\delta^2} \frac{\sqrt{3}}{2} \int_0^1 r^4 \omega(r)  \frac{r \delta | \xx | /2- \sin (r \delta | \xx | /2)}{r^3 \delta |\xx|} dr >\frac{4 \pi}{\delta^2} \frac{\sqrt{3}}{2} \frac{6}{96} > \frac{1}{16 \ \delta^2}.
\end{equation}

As for $\delta |\xx| \leq 1$, we apply the following inequality
\[
1- \frac{x^2}{2} \leq \cos(x) \leq 1-\frac{x^2}{2}+\frac{x^4}{24}, \qquad \forall \ 0<x<1,
\]
to get
\begin{equation} \label{a14}
\begin{split}
\lambda_{\delta}(\xx) = \frac{4 \pi}{\delta^2} \int_0^{\pi/2} \sin (\phi) \int_0^1 r^2 \omega(r) (1-\cos(r \ \cos(\phi) \delta |\xx|)) dr d\phi \\
\geq  \frac{4 \pi}{2} |\xx|^2 \int_0^{\pi/2} \cos^2(\phi) \sin (\phi) d \phi \int_0^1 r^4 \omega(r) dr \\
- \ \frac{4 \pi}{24} \ |\xx|^2 \ ( \delta |\xx|)^2 \int_0^{\pi/2} \cos^4(\phi) \sin(\phi) d \phi \int_0^1 r^6 \omega(r) dr \\
\geq  |\xx|^2- \frac{1}{20} |\xx|^2 \ ( \delta |\xx|)^2 \geq \frac{1}{2} |\xx|^2,
\end{split}
\end{equation}
and
\begin{equation} \label{a23}
\begin{split}
\lambda_{\delta}(\xx) = \frac{4 \pi}{\delta^2} \int_0^{\pi/2} \sin (\phi) \int_0^1 r^2 \omega(r) (1-\cos(r \ \cos(\phi) \delta |\xx|)) dr d\phi \\
\leq  \frac{4 \pi}{2} |\xx|^2 \int_0^{\pi/2} \cos^2(\phi) \sin (\phi) d \phi \int_0^1 r^4 \omega(r) dr =|\xx|^2.
\end{split}
\end{equation}

The estimates $\eqref{a13}$ and $\eqref{a14}$ give a uniform lower bound for $\lambda_{\delta}(\xx)$.

By mimicing the calculation of $\lambda_{\delta}(\xx)$, the bound on $c_{\delta}(\xx)$ can be obtained similarly.

The bound on $|\bb_\delta(\xx)|$ is also easy to get.
For $\delta |\xx|<1$, notice that
\[
\sin(x) \geq x- \frac{x^3}{6}  \qquad \forall \ 0<x<1,
\]
then
\begin{equation} \label{a16}
\begin{split}
|\bb_{\delta}(\xx)| =&4 \pi \int_0^{\pi/2} \cos(\phi) \sin (\phi) \ \int_0^{\delta} r^2 \hat{\omega}_{\delta}(r) \ \sin(r \ \cos(\phi) |\xx|) \ dr d\phi  \\
=&\frac{4 \pi}{\delta} \int_0^{\pi/2} \cos(\phi) \sin (\phi) \ \int_0^{1} r^2 \hat{\omega}(r) \ \sin(r \ \cos(\phi) \delta |\xx|) \ dr d\phi  \\
  \geq& 4 \pi |\xx| \int_0^{\pi/2} \cos^2(\phi) \sin (\phi) d \phi \ \int_0^1 r^3 \hat{\omega}(r) dr  \\
&-  \frac{4 \pi |\xx|}{6}  (\delta |\xx|)^2 \int_0^{\pi/2} \cos^4(\phi) \sin (\phi) d \phi \ \int_0^1 r^5 \hat{\omega}(r) dr \\
 \geq &|\xx|-\frac{(\delta |\xx|)^2}{10} |\xx| \geq \frac{1}{2} |\xx|.
\end{split}
\end{equation}

\section{Proof of \eqref{a42}}
For $\delta |\xx|>1$, we have
\begin{equation} \label{a40}
\begin{split}
\lambda_{\delta}(\xx) \geq& \frac{4\pi}{\delta^2} \int_0^{\pi/3} \sin (\phi) \int_0^1 r^2 \omega(r) (1- \cos (r \ \cos (\phi) \delta |\xx|)) dr \ d\phi \\
\geq& \frac{4\pi}{\delta^2} \int_0^{\pi/3} \sin (\phi) \int_0^1 \frac{C_1}{r^{1+2 \alpha}} (1- \cos (r \ \cos (\phi) \delta |\xx|)) dr \ d\phi \\
\geq& \frac{4\pi (\delta |\xx|)^{2 \alpha}}{\delta^2} \int_0^{\pi/3} \sin (\phi) \ \cos^{2 \alpha} (\phi) \int_0^{\cos(\phi) \delta |\xx|} \frac{C_1}{r^{1+2 \alpha}}(1-\cos(r)) dr \ d\phi \\
\geq& \frac{4\pi (\delta |\xx|)^{2 \alpha}}{\delta^2} \int_0^{\pi/3} \sin (\phi) \ \cos^{2 \alpha} (\phi) \ d \phi \ \int_0^{ 1 /2} \frac{C_1}{r^{1+2 \alpha}}(1-\cos(r)) dr \\
\geq& \frac{C_1}{6}  \frac{1}{\delta^{2-2 \alpha }} |\xx|^{2 \alpha}.
\end{split}
\end{equation}
Meanwhile, for $\delta |\xx| \leq 1$, we obtain from $\eqref{a14}$ that $$\lambda_{\delta}(\xx) \geq \frac{1}{2}|\xx|^2 \geq \frac{|\xx|^{2-2\alpha}}{2} |\xx|^{2 \alpha}.$$

Since $\delta \in (0,1)$ and $\xx \in \mathbb{Z}^d$, we see $\lambda_{\delta}(\xx)$ has a uniform lower bound $C \ |\xx|^{2 \alpha}$ in either way, where the constant $C=\min \ \{ \ 1, \frac{C_1}{6} \}$. Using such bound, we are able to control the coefficients in $\eqref{a8}$ and $\eqref{a9}$. First,
\begin{equation}
\left| \frac{1}{\lambda_{\delta}(\xx)} \ \frac{\bb_{\delta}(\xx) \otimes \bb_{\delta} (\xx)}{| \bb_{\delta} (\xx) |^2+ \delta^2  c_{\delta}(\xx) \lambda_{\delta}(\xx)  } \right| = \frac{1}{\lambda_{\delta}(\xx)} \ \frac{| \bb_{\delta} (\xx) |^2}{| \bb_{\delta} (\xx) |^2+ \delta^2 c_{\delta}(\xx) \lambda_{\delta}(\xx)} \leq  \frac{1}{C}  |\xx|^{-2\alpha}.
\end{equation}  
In addition, 
 for $\delta |\xx|>1$,
\begin{equation}
\begin{split}
 \frac{ | (\bb_{\delta}(\xx))^T |}{| \bb_{\delta} (\xx) |^2+ \delta^2 c_{\delta}(\xx) \lambda_{\delta}(\xx) } \leq \frac{ | (\bb_{\delta}(\xx))^T | }{ 2 \delta \sqrt{c_{\delta}(\xx) \lambda_{\delta}(\xx)} \ | \bb_{\delta} (\xx) |  }=\frac{1}{2 \delta \  \sqrt{c_{\delta}(\xx) \lambda_{\delta}(\xx)} } \\
 \leq \frac{1}{2 \delta \sqrt{\frac{m}{16 \delta^2} C |\xx|^{2 \alpha}}} = \frac{2}{\sqrt{C \ m}} |\xx|^{-\alpha},
 \end{split}
\end{equation}
while for $\delta |\xx|<1$, 
\begin{equation} 
 \frac{ | (\bb_{\delta}(\xx))^T |}{| \bb_{\delta} (\xx) |^2+ \delta^2 c_{\delta}(\xx) \lambda_{\delta}(\xx)  } \leq \frac{1}{| \bb_{\delta} (\xx) |} \leq \frac{2}{ |\xx|} \leq 2 |\xx|^{-\alpha} .
\end{equation}

Finally, we need to bound the term $\frac{ \lambda_{\delta}(\xx)}{
q_\delta(\xx)
}$ above. Combining all the above bounds and return to $\eqref{a8}$ and $\eqref{a9}$, we see that
\begin{equation} 
\begin{split}
| \hat{\uu}_{\delta}(\xx) | \leq C( \ |\xx^{-2\alpha}| \ |\hat{\f}(\xx)|  \ + \  | \xx^{-\alpha}|  \ |\hat{g}(\xx) | \ ),  \\
| \hat{p}_{\delta}(\xx)| \leq C( \ |\xx^{-\alpha}| \ |\hat{\f}(\xx)|  \ + \  |  \hat{g}(\xx) | \ ),
\end{split}
\end{equation}
here the constant $C$ depends only on $C_1$. Hence we have $\eqref{a42}$.

\section{Proof of the second part of Theorem \ref{a39}}
Again, we divide into two cases. First, 
for $\delta |\xx| \leq 1$, by following $\eqref{a18}$ and $\eqref{a19}$ we deduce 
\begin{equation}
\begin{split}
 | \xx |^{\alpha} \bigg| (\frac{1}{\lambda_{\delta}(\xx)}- \frac{1}{|\xx|^2} )I_d-( \frac{1}{\lambda_{\delta}(\xx)} \frac{\bb_{\delta}(\xx) \otimes \bb_{\delta} (\xx)}{| \bb_{\delta} (\xx) |^2+ \delta^2  c_{\delta}(\xx) \lambda_{\delta}(\xx)   } - \frac{ \xx \otimes \xx }{|\xx|^4 }) \bigg| \\
 <(4M+2) \delta^2 | \xx |^{\alpha}  \leq (4M+2) \delta^{ \min \{ 2, 2-2\alpha+\eta \} } |\xx|^{\eta-\alpha}.
\end{split}
\end{equation}
Moreover, recall that $\eqref{a20}$, for any $\eta \in [0, 1+\alpha]$ we have
\begin{equation}
 | \frac{(\bb_{\delta}(\xx))^T}{q_\delta(\xx)} - \frac{\xx^T}{|\xx|^2} | \leq (8M+1) \delta^2 |\xx|  \leq (8M+1) \delta^{1-\alpha+\eta} |\xx|^{\eta-\alpha},
\end{equation}
and
\begin{equation}
 | \xx |^{\alpha} \ | \frac{(\bb_{\delta}(\xx))^T}{q_\delta(\xx)} - \frac{\xx^T}{|\xx|^2} | \leq (8M+1) \delta^2 |\xx| \ |\xx|^{\alpha} \leq (8M+1) \delta^{1-\alpha+\eta} |\xx|^{\eta}.
\end{equation}
Finally, recall that $\eqref{a35}$, for any $\beta \in [0,2]$ we deduce
\begin{equation}
 | \frac{ \lambda_{\delta}(\xx) }{q_\delta(\xx)}  - 1 | \leq (4M+1) \delta^{\beta} |\xx|^{\beta}.
 \end{equation}

Next, for $\delta |\xx|>1$, since we obtain the strengthened lower bound for $\lambda_{\delta}(\xx)$ from $\eqref{a40}$, we can plug it into $\eqref{a30}$ and see
\begin{equation}
\begin{split}
|\xx|^{\alpha} \ \left|(\frac{1}{\lambda_{\delta}(\xx)}- \frac{1}{|\xx|^2}) I_d - ( \frac{1}{\lambda_{\delta}(\xx)} \frac{\bb_{\delta}(\xx) \otimes \bb_{\delta} (\xx)}{| \bb_{\delta} (\xx) |^2+ \delta^2  c_{\delta}(\xx) \lambda_{\delta}(\xx)  }  - \frac{1}{|\xx|^2} \frac{ \xx \otimes \xx }{|\xx|^2 }) \right| \\
\leq 2 |\xx|^{\alpha} (\frac{1}{\lambda_{\delta}(\xx)} + \frac{1}{|\xx|^2}) < 2 |\xx|^{\alpha} ( \frac{6}{C_1} \delta^{2-2 \alpha } |\xx|^{-2 \alpha} + |\xx|^{-2}) \\
< 2C \ \delta^{2-2\alpha} |\xx|^{-\alpha} < 2C \ \delta^{ \min \{ 2, 2-2\alpha+\eta \} } |\xx|^{\eta-\alpha} .
\end{split}
\end{equation}
Now, for any $\eta \in [0, 1+\alpha]$, we plug $\eqref{a40}$ again into $\eqref{a31}$ and get
\begin{equation}
\begin{split}
| \frac{(\bb_{\delta}(\xx))^T}{q_\delta(\xx) } - \frac{\xx^T}{|\xx|^2} | \leq  \frac{1}{2 \delta  \sqrt{\lambda_{\delta}(\xx)  c_{\delta}(\xx)} } + \frac{1}{|\xx|} \\
 \leq (2 \delta \sqrt{\frac{C_1}{6} \delta^{2 \alpha-2 } |\xx|^{2 \alpha}  \frac{m}{16 \delta^2}} \ )^{-1} + \frac{1}{|\xx|} \leq  C  \delta^{1-\alpha} |\xx|^{-\alpha}  \leq C  \delta^{1-\alpha+\eta} |\xx|^{-\alpha+\eta}.
\end{split}
\end{equation}
Therefore,
\begin{equation}
|\xx|^{\alpha} | \frac{(\bb_{\delta}(\xx))^T}{q_\delta(\xx) } - \frac{\xx^T}{|\xx|^2} | \leq  C \delta^{1-\alpha} \leq C \delta^{1-\alpha+\eta} |\xx|^{\eta}.
\end{equation}
Fnally, recall that $\eqref{a36}$, for any $\beta \in [0,2]$ we have 
\begin{equation} 
| \frac{ \lambda_{\delta}(\xx)}{q_\delta(\xx)}  - 1 | \leq 1+\frac{16}{m} \leq (1+\frac{16}{m} ) \delta^{\beta} |\xx|^{\beta}.
\end{equation}

Hence we combine the two cases above and return to the equation $\eqref{a28}$  and $\eqref{a29}$ to obtain the estimate $\eqref{a46}$ and $\eqref{a47}$.
